\documentclass[12pt,a4paper]{amsart}
\usepackage[T1]{fontenc}
\usepackage[utf8]{inputenc}
\usepackage[a4paper,centering,marginpar=25mm]{geometry}
\usepackage[foot]{amsaddr}
\usepackage{amssymb}
\usepackage{esint}
\usepackage[hyperfootnotes=false,colorlinks,breaklinks,urlcolor=blue,linkcolor=blue,citecolor=blue]{hyperref}
\usepackage[nameinlink,capitalise]{cleveref}
\usepackage{comment}
\usepackage[english]{babel}
\newtheorem{thm}{Theorem}[section]
\newtheorem{lemma}[thm]{Lemma}
\newtheorem{prop}[thm]{Proposition}

\theoremstyle{definition}
\newtheorem{defn}[thm]{Definition}
\theoremstyle{remark}

\newtheorem{oss}[thm]{Remark}
\numberwithin{equation}{section}

\newcommand{\N}{\mathbb{N}}
\newcommand{\Z}{\mathbb{Z}}
\newcommand{\Q}{\mathbb{Q}}
\newcommand{\R}{\mathbb{R}}
\newcommand{\T}{\mathbf{T}^3}

\newcommand{\bg}{\bigg}
\newcommand{\Axi}{A_{\xi}}
\newcommand{\xiA}{\xi \times \Axi}
\newcommand{\BW}{\mathbf{W}_{\xi}}
\newcommand{\BWpr}{\mathbf{W}_{\xi'}}
\newcommand{\et}{\eta_{\xi}}
\newcommand{\am}{a_{\xi}}
\newcommand{\gammxi}{\gamma_{\xi}(R)}
\newcommand{\AntiDiv}{\mathcal{R}}
\newcommand{\proj}{\mathbf{P}}
\newcommand{\LinfL}{L^{\infty}_tL^1_x}
\newcommand{\LinfLSec}{L^{\infty}_tL^2_x}

\newcommand{\Diss}{(-\Delta)^{\frac{5}{4},\beta}}
\DeclareMathOperator{\supp}{Supp}
\DeclareMathOperator{\Div}{div}
\DeclareMathOperator{\sign}{sgn}
\DeclareMathOperator{\Tr}{Tr}
\newcommand{\LlL}{\mathrm{L}(\mathrm{logL})}

\begin{document}
\title[]{Non-uniqueness of weak solutions for a logarithmically supercritical hyperdissipative Navier-Stokes system}
\author[M. Romito]{Marco Romito}
  \address{Dipartimento di Matematica, Universit\`a di Pisa, Largo Bruno Pontecorvo 5, I--56127 Pisa, Italia}
  \email{\href{mailto:marco.romito@unipi.it}{marco.romito@unipi.it}}
\author[F. Triggiano]{Francesco Triggiano}
  \address{Scuola Normale Superiore, Piazza dei Cavalieri, 7, 56126 Pisa, Italia}
  \email{\href{mailto:francesco.triggiano@sns.it}{francesco.triggiano@sns.it}}
\subjclass[2020]{Primary 35Q30; secondary 76D03, 76D05}
\keywords{Navier-Stokes equations, convex integration, non-uniqueness, slightly supercritical hyper-dissipation}
\date{June 11, 2024}
\begin{abstract}
The existence of non-unique solutions of finite kinetic energy for the three dimensional Navier-Stokes equations is proved in the slightly supercritical hyper-dissipative setting introduced by Tao \cite{T10}.
The result is based on the convex integration techniques of Buckmaster and Vicol \cite{BV19}, and extends Luo and Titi \cite{LT20} in the slightly supercritical setting. To be able to be closer to the threshold identified by Tao, we introduce the impulsed Beltrami flows, a variant of the intermittent Beltrami flows of Buckmaster and Vicol.
\end{abstract}
\maketitle

\section{Introduction}

In this work we consider the following hyperdissipative Navier-Stokes system on the 3-dimensional torus
\begin{equation}\label{MainEq}
  \begin{cases}
    \partial_tu+\Div(u\otimes u)+\nu (-\Delta)^{\theta,\beta}u+\nabla p=0,\\
    \Div(u)=0,
  \end{cases}  
\end{equation}
with periodic boundary conditions and zero spatial mean, where $\beta$ and $\theta$ are non-negative parameters, and the dissipative term is defined via Fourier transform $\mathcal{F}$ by
    \begin{equation*}
      \mathcal{F}[(-\Delta)^{\theta,\beta}u](k)=\mathbf{1}_{\{k\neq 0\}}(k)\frac{|k|^{2\theta}}{\log^{\beta}(10+|k|_1)}\mathcal{F}(u)(k),\qquad\text{for all } k \in \Z^3, 
    \end{equation*}
    where $|k|$ and $|k|_1$ are the 2-norm and the 1-norm of $k$, respectively.

 Well-posedness of a similar system has been widely studied in the last century. In particular, it is well known that global existence and uniqueness of sufficiently smooth solutions holds when $\beta=0$ and $\theta\geq\frac{5}{4}$ \cite{L69,KP02}.
 
On the critical threshold $\theta=\frac{5}{4}$, the same result has been proved by Tao \cite{T10} and Barbato et al. \cite{BMR15} when the dissipative term is slightly weakened by a Fourier multiplier that mimics a logarithmic behaviour raised to a sufficiently small power $\beta$.

When $\theta<\frac{5}{4}$ it is well known that there are various open problems, such as uniqueness of Leray-Hopf solutions for classical NSE, i.e. $\theta=1$ and $\beta=0$.

Furthermore, various negative results appeared during the last decade.
In particular, Luo et al. \cite{LT20} proved non-uniqueness of global distributional solutions, less regular than Leray weak solutions \cite{L34}, by means of the convex integration theory \cite{BV19,BV20,DS14,I18,MS18}.
The convex integration machinery has been used by Li et al. \cite{LQZZ22} in order to prove non-uniqueness in suitable functional spaces also above the Lions exponent, i.e. $\theta \geq \frac{5}{4}$ and $\beta=0$, but non-uniqueness is proved in the space $C_tL^p$ for all $1\le p<2$, thus in larger spaces than energy (in general below the Ladyzenskaja-Prodi-Serrin threshold).

In this paper we work at the critical threshold $\theta=\frac{5}{4}$, whose relevance can be also highlighted by making some scaling considerations. See \cite{K10} for a more complete discussion on the topic.
Under the assumption $\beta=0$, 
we recall that a functional space $\mathbf{X}$ is said to be a critical space if $\|u\|_{\mathbf{X}}$\ is invariant under the following scaling transformation:
\begin{equation*}
    u(t,x)\to \lambda^{2\theta-1}u(\lambda^{2\theta}t,\lambda x),\qquad
    p(t,x)\to \lambda^{4\theta-2}p(\lambda^{2\theta}t,\lambda x).
\end{equation*}
Since its norm represents the kinetic energy, one important functional space is $\mathbf{X}=C_tL^2_x$, which is a critical space only if $\theta=\frac{5}{4}$.

In this work we show that the choice $\theta=\frac54$ and $\beta$ large enough, that is weakening the dissipation by a logarithmic corrector, allows to prove a non-uniqueness result for distributional solutions to \cref{MainEq} that are in the relevant functional space $C_tL^2_x$.
This non-uniqueness result in a logarithmic supercritical case is obtained by stretching as much as possible the convex integration scheme proposed in \cite{BV19,LT20}.
In particular from a technical point of view, the proof is mainly based on combining new estimates for various Fourier multipliers, which allow to exploit the logarithmic corrector, and classical harmonic analysis results in some Orlicz spaces. Moreover, we also propose a new modification of the classical Beltrami flows to achieve the result by assuming a weaker condition on $\beta$. For more details see the discussion in \cref{r:beta}.

The paper is organized as follows. In \cref{s:main} we state and prove the main result by means of the iteration lemma. In \cref{c:FourierMult} we introduce and study various Fourier multipliers. \cref{c:IterLemma} is devoted to the proof of the iteration lemma, namely the inductive step which solves \cref{MainEq} in the limit.

\subsection{Notations}
To ease the notations, for any $\alpha\geq 0$ and $p \in [1,\infty]$ we denote
\begin{align*}
     C_tL^p_x
       &= C([0,\infty),L^p(\T,\R ^3)),&
     C_t\LlL^{\alpha}_x
       &= C([0,\infty),\LlL^{\alpha}(\T,\R ^3)),\\
     L^{\infty}_tL^p_x
       &= L^{\infty}([0,\infty),L^p(\T,\R ^3)),&
     L^{\infty}_t\LlL^{\alpha}_x
       &= L^{\infty}([0,\infty),\LlL^{\alpha}(\T,\R ^3)),
\end{align*}
where $\T$ denotes the 3-dimensional torus and $\LlL^{\alpha}$ is the Orlicz space associated with the Young function $A(s)=s\log^{\alpha}(2+s)$ (see \cref{DefnLlL} for the detailed definition).
We denote by $C_w(\R,L^2(\T,\R^3))$  the space of weakly continuous functions with values within the Banach space $L^2(\T,\R^3)$, while $C^{\infty}_c(\R,C^{\infty}(\T,\R^3))$ is the space of smooth compactly supported and spatially periodic functions. In addition, if $v \in C^{\infty}_c(\R,C^{\infty}(\T,\R^3))$ we denote its time support by $\supp_t(v)$.
For any sub-interval $I\subset \R_+$, let $\proj_I$ be the operator defined by
\begin{equation*}
    \proj_Ig(x)=\sum_{|k|_1\in I}\mathcal{F}(g)(k)e^{ik\cdot x}.
\end{equation*}

We denote by $P_H$ the Leray-Helmholtz projector, namely the projector onto the space of divergence-free vector fields.
For any $A\subset \R$ and $\delta>0$, let $N_{\delta}(A)$ be the $\delta$-neighborhood of A, namely
\begin{equation*}
    N_{\delta}(A)=\{z \in \R: \text{ there is }y \in A\text{ s.t. }|z-y|<\delta\}.
\end{equation*}

If $x,y\in \R^3$, then $x\times y$ will denote the standard vector cross-product between $x$ and $y$. Given $a,b \in \R$, we use the notations $a\lesssim b$ and $a\approx b$ if there exist $C, C_1,C_2>0$ such that $a\le C b$ and  $C_1a\le b\le C_2a$, respectively.

\section{Main result}\label{s:main}

We first recall the definition of weak solution to \cref{MainEq}.
\begin{defn}
    A vector field $v \in C_w(\R,L^2(\T,\R^3))$ is a weak solution to \cref{MainEq} if it is satisfied in a distributional sense, namely,
    \begin{equation*}
        \int_{\R}\int_{\T}(\partial_t \phi )\cdot v\,dx\,dt=\int_{\R}\int_{\T}v\cdot(v\cdot \nabla)\phi -\big(\nu (-\Delta)^{\theta,\beta}\phi\big)\cdot v\,dx\,dt,
    \end{equation*}
    for all $\phi \in C^{\infty}_c(\R,C^{\infty}(\T,\R^3))$ such that $\Div(\phi(t,\cdot))=0$ for all $t\in \R$.
\end{defn}

\begin{thm}\label{MainThm}
    For $\theta=\frac54$ and $\beta>\frac{29}{2}$, consider a smooth divergence-free vector field $u:[0,\infty) \times \T \to \mathbb{R}^3$, with compact support in time and zero spatial-mean,
    \begin{equation*}
        \int_{\T}u(t,x)dx=0.
    \end{equation*}
    Then, for all $\epsilon>0$, there exists a weak solution $v\in C(\R,L^2(\T,\R^3))$ of \cref{MainEq} with compact support in time and such that
    \begin{equation*}
        \||\nabla|^{\frac{3}{2}}T_M(v-u)\|_{\LinfL}\le \epsilon,
    \end{equation*}
    where $T_M$  and $|\nabla|^{\frac{3}{2}}$ are the operators corresponding to the Fourier multipliers $M(k)=\mathbf{1}_{\{k\neq 0\}}(k)\log^{-\beta}(|k|_1+10)$ and $m'(k)=|k|^{\frac{3}{2}}$ (see \cref{FourMult}).
    In particular, there are infinitely many weak solutions with zero initial condition.
\end{thm}

The result is shown by using the classical iterative procedure of the convex integration theory, i.e. following the groundbreaking work of Buckmaster and Vicol \cite{BV19}.
Therefore, the proof is strongly based on the following lemma.
\begin{lemma}[Iteration Lemma]\label{IterLemma}
    Fix $\theta=\frac54$ and $\beta>\frac{29}{2}$. let $(v_q,p_q,R_q)$ be a smooth solution to the approximate system
    \begin{equation}\label{ApproxEq}
      \begin{cases}
        \partial_tu+\Div(u\otimes u)+\nu (-\Delta)^{\frac{5}{4},\beta}u+\nabla p=\Div(R),\\
        \Div(u)=0,\\
        \Tr(R(t,x))=0,\qquad
          R(t,x)=R(t,x)^T\quad
          \text{for all }t,x,
      \end{cases}
    \end{equation}
    such that 
    \begin{equation*}
        \|R_q\|_{\LinfL}\le \delta_{q+1}.
    \end{equation*}
    Then for every $\delta_{q+2}>0$, there exists a smooth solution $(v_{q+1},p_{q+1},R_{q+1})$ of \cref{ApproxEq} such that
    \begin{equation}\label{IterLemma1}
      \begin{gathered}
        \|R_{q+1}\|_{\LinfL}\le \delta_{q+2},\\
        \supp_t (v_{q+1}) \cup \supp_t (R_{q+1}) \subset N_{\delta_{q+1}} (\supp_t (v_{q}) \cup \supp_t (R_{q})).
      \end{gathered}
    \end{equation}
    Moreover, the following estimates hold
    \begin{equation}\label{IterLemma2}
      \begin{aligned}
        \|v_{q+1}-v_q\|_{\LinfLSec}
          &\le C \delta_{q+1}^{\frac{1}{2}},\\
        \||\nabla|^{\frac{3}{2}}T_M(v_{q+1}-v_q)\|_{\LinfL}
          &\le \delta_{q+2}.
      \end{aligned}
    \end{equation}
\end{lemma}

\begin{oss}[On the parameter $\beta$]\label{r:beta}
The requirement on the parameter $\beta$ will be clear in the proof of \cref{IterLemma}. In particular, it is required that $\beta>\frac{29}{2}$. This specific value is due to various harmonic analysis estimates (in negative, forcing a larger value of $\beta$), which are mainly shown in \cref{c:FourierMult}, and to a suitable construction of the building blocks of the convex integration scheme (in positive, lowering the minimal threshold we impose on $\beta$), which differs from the intermittent Beltrami flows \cite{BV19} by a logarithmic factor (see \cref{DiffBB}).

As already mentioned in the introduction, Tao \cite{T10} and Barbato et al. \cite{BMR15} proved the existence and uniqueness of smooth solutions to an equation similar to \cref{MainEq}, where the dissipative term $D(u)$ is defined via Fourier transform as 
\begin{equation*}
    \mathcal{F}[D(u)](k)=\mathbf{1}_{\{k\neq 0\}}(k)\frac{|k|^{2\theta}}{g(|k|)}\mathcal{F}(u)(k),\qquad\text{for all } k \in \Z^3, 
\end{equation*}
for any positive, non-decreasing function $g$ such that $\int_1^{\infty}\frac{1}{sg(s)}ds=\infty$. This operator differs from our dissipative term in two aspects.
Firstly, it presents a corrector computed on $|k|$, while our works on $|k|_1$.
This difference is due to the fact that we need to estimate the $\mathcal{L}(L^1,L^1)$-norm of the operator associated with the Fourier multiplier $log^{-\beta}(10+|k|_1)\proj_{\geq N}$ (see \cref{Div-1} and \cref{Div-1Log}) in order to exploit completely the Fourier properties of our building blocks.
The second difference is the presence of the function $g$. Even if their result covers a larger family of dissipative terms, the most interesting one is clearly given by $g(s)=log^{\beta}(10+|k|)$ with $\beta\le 1$. 

Therefore, apart from the structural differences between our dissipative term and the one adopted in \cite{T10,BMR15}, it can be interesting to study why their constraint on $\beta$ is adimensional while ours depends on the dimension (see \cref{Div-1} and \cref{Div-1Log}) and to understand if additional technical novelties could allow to weaken our assumption on the $\beta$ parameter in order to get closer to the threshold identified by their results.
\end{oss}

Assuming the previous Lemma, the proof of \cref{MainThm} basically follows the lines of Luo and Titi \cite{LT20}. We report it for completeness.

\begin{proof}[Proof of \cref{MainThm}]
    Consider the triple
    \begin{gather*}
        v_0=u,\qquad
        p_0=-\frac{1}{3}|v_0|^2,\\
        R_0=\mathcal{R}(\partial_tv_0+\nu(-\Delta)^{\frac{5}{4},\beta}v_0)+v_0\otimes v_0+p_0I,
    \end{gather*}
    where $\mathcal{R}$ is the anti-divergence operator, defined below in \cref{AntiDivDefn}.

    Given $\delta_1=\|R_0\|_{\LinfL}$ and $\delta_q=2^{-q}\epsilon$, consider the sequence $(v_q,p_q,R_q)_{q \in \N}$ coming from \cref{IterLemma}.
    Then,
    \begin{gather*}
      \sum_q \|v_{q+1}-v_q\|_{\LinfLSec}\lesssim \sum_{q}\delta_{q+1}^{\frac{1}{2}}< \infty,\\
      \|R_q\|_{\LinfL}\to 0\qquad \text{as }q \to \infty.
    \end{gather*}
    Therefore, $(v_q)_q$ converges strongly to a weak solution $v\in C_TL^2_x$. Moreover,
    \begin{gather*}
         \||\nabla|^{\frac{3}{2}}T_M(v-u)\|_{\LinfL}\le \sum \||\nabla|^{\frac{3}{2}}T_M(v_{q+1}-v_q)\|_{\LinfL}\le \sum_q \delta_{q+1}\le \epsilon,\\
         \supp_t(v)\subset \cup_q \supp_t (v_q)\subset N_{\sum_q \delta_{q+1}}(\supp_t(u))\subset N_{\delta_1+\epsilon}\supp_t(u).
    \end{gather*}
    
    Finally, infinitely many weak solutions with zero initial condition can be obtained by considering $u(t,x)=\phi(t)\sum_{|k|\le N}a_ke^{ik\cdot x}$ with $a_0=0$, $a_k^*=a_{-k}$, $a_k\cdot k=0$ and $\phi \in C^{\infty}(\R _+)$, with $\phi(0)=0$.
\end{proof}

\section{Preliminaries on Fourier multipliers}\label{c:FourierMult}
In this section we mainly prove that the operators associated with various Fourier multipliers are bounded from $\LlL^{\beta}$ to $\LlL^{\alpha}$ for suitable $\alpha,\beta \geq 0$.
Let us start by defining the functional space $\LlL^{\alpha}$ and by proving an elementary inequality.

\begin{defn}\label{DefnLlL}
Given $0\le\alpha <\infty$ and $d\in \N$, consider the space 
\begin{equation*}
    \LlL^{\alpha}=\Bigl\{f:\T\to\R^d\text{ measurable }:\int_{\T}|f(x)|\log^{\alpha}(2+|f(x)|)\,dx < \infty\Bigr\}.
\end{equation*}
This is a Banach space when equipped with the so-called \emph{Luxemburg norm} 
    \begin{equation}\label{LuxDefn}
        \|f\|_{\LlL^{\alpha}}=\inf\Bigl\{\lambda >0: \int_{\T}A(|f(x)|/\lambda)dx\le 1\Bigr\},
    \end{equation}
    where $A(s)=s\log^{\alpha}(2+s)$ for all $s\geq 0$.
\end{defn}

\begin{lemma}\label{PropLuxemburg}
  Consider a scalar function $f \in L^{\infty}(\T)$ and $g \in \LlL^{\alpha}(\T,\R^d)$ for $\alpha\geq 0$. Then
  \begin{equation*}
      \|fg\|_{\LlL^{\alpha}}\le \|f\|_{L^{\infty}}\|g\|_{\LlL^{\alpha}}.
  \end{equation*}
\end{lemma}
\begin{proof}
    If $A(s)=s\log^{\alpha}(2+s)$, then
    \begin{equation*}
        \int_{\T} A\Big(\frac{|fg|}{\|f\|_{L^{\infty}}\|g\|_{\LlL^{\alpha}}}\Big)dx\le \int_{\T} A\Big(\frac{|g|}{\|g\|_{\LlL^{\alpha}}}\Big)dx\le 1,
    \end{equation*}
    where the first inequality follows since $A$ is non-decreasing, and the second follows from the definition of the Luxemburg norm.
\end{proof}

Now, we first define the operator corresponding to a Fourier multiplier. Then we will focus on specific multipliers. 
\begin{defn}\label{FourMult}
    Given a family $m=(m(k))_{k \in \Z^3}$, its corresponding operator $T_m$ is defined by
    \begin{equation*}
        T_m(f)(x)=\sum_{k \in \Z^3}m(k)\mathcal{F}(f)(k)e^{ik\cdot x},\qquad \text{ for } f \in C^{\infty}(\T),
    \end{equation*}
    where $\mathcal{F}(f)(k)$ is the $k^{\text{th}}$ Fourier coefficient of $f$.
\end{defn}

\begin{prop}\label{Div-1}
    Consider $m(k)=\mathbf{1}_{\{k\neq 0\}}(k)\frac{1}{|k|_1}$, then $T_m$ is bounded from $L^1(\T)$ to itself.

    Moreover, for all $N\in \N$,
    \begin{equation}\label{ProjEstimateDiv-1}
        \|\proj_{|k|_1> N}T_m\|_{\mathcal{L}(L^1,L^1)}\lesssim \frac1N\log^3(N),
    \end{equation}
    where $\mathcal{L}(L^1,L^1)$ is the space of bounded and linear operators from $L^1(\T)$ to itself.
\end{prop}
\begin{proof} 
    Consider $f_N(x)=\sum_{|k|_1\le N}\frac{1}{|k|_1}e^{ik\cdot x},$ then the following formula holds,
    \begin{equation*}
        f_N(x)=\sum_{j=1}^{N-1}D^1_j(x)\Bigl(\frac{1}{j}-\frac{1}{j+1}\Bigr)+D^1_N(x)\frac{1}{N},
    \end{equation*}
    where $D^1_j(x)=\sum_{|k|_1\le j}e^{ik\cdot x}$.

    Therefore, for $N<M$,
    \begin{equation*}
        \|f_N-f_M\|_{L^1}\le \frac1N\|D_N^1\|_{L^1}+\frac1M\|D_M^1\|_{L^1}+\sum_{j=N}^{M-1}\|D^1_j\|_{L^1}\big(\frac{1}{j}-\frac{1}{j+1}\big). 
    \end{equation*}
    It is known that $\|D^1_j\|_{L^1}\lesssim \log^3(j)$ (see \cite{KL21}), so
    \begin{equation*}
        \|f_N-f_M\|_{L^1}\le \frac1N\log^3(N)+\frac1M\log^3(M)+\sum_{j=N}^{M-1}\frac{\log^3(j)}{j(j+1)}\lesssim \frac1N\log^3(N).
    \end{equation*}
    The sequence $(f_N)_N$ is a Cauchy sequence in $L^1$ and converges to $f$ such that $\nu(A)=\int_Af(x)\,dx$ is a finite Borel measure on the torus which has $m(k)$ as Fourier coefficients. The existence of the measure $\nu$ implies that $m$ is a $(1,1)$-multiplier and that $\|T_m\|_{\mathcal{L}(L^1,L^1)}\le |\nu|=\|f\|_{L^1}$, where $|\nu|$ is the total variation of the measure $\nu$ (see \cite{G08}).
    At last, by the same argument, the estimate on $\|\proj_{|k|_1> N}T_m\|_{\mathcal{L}(L^1,L^1)}$ follows from
    \begin{equation*}
        \|f-f_N\|_{L^1}=\lim_{M\to \infty}\|f_M-f_N\|_{L^1}\le \frac1N\log^3(N).\qedhere
    \end{equation*}
\end{proof}
\begin{oss}\label{Div-1Log}
        By repeating the same argument above, the conclusions of \cref{Div-1} hold for $M(k)=\mathbf{1}_{k\neq 0}(k)\frac{1}{\log^{\beta}(|k|_1+10)}$, with $\beta>3$, and 
        \begin{equation}\label{ProjEstimateDiv-1log}
            \|\proj_{|k|_1> N}T_M\|_{\mathcal{L}(L^1,L^1)}\lesssim \frac{\log^3(N)}{\log^{\beta}(10+N)}.
        \end{equation}
        
        We notice though that the same machinery works for $m(k)=\frac{1}{|k|_{\infty}}$, but not for $m(k)=\frac{1}{|k|}$ because of $\|D_j^2\|_{L^1}\approx j^{\frac{3-1}{2}}$ where $D^2_j(x)=\sum_{|k|\le j}e^{ik\cdot x}$, see for instance \cite{S75}.
        
        These results clearly hold in any dimension $d$ and make use of a multidimensional version of Abel summation formula. In addition, the appearance of a factor $\log^{d}(N)$ is coherent with various results obtained for $L^1$-convergence of double Fourier series \cite{M94}.
\end{oss}

\begin{prop}\label{DirHilbert}
    Given $j\in \{1,2,3\}$ and the multiplier $m(k)=\mathbf{1}_{\{k\neq 0\}}(k)i[\sign(k_j)]$, then $T_m$ is bounded from $\LlL^{\alpha+1}$ to $\LlL^{\alpha}$ for all $\alpha\geq 0$.
\end{prop}

\begin{proof}
    It is known that if an operator $T$ is bounded from $L^2$ to itself and from $L^1$ to the weak $L^1$ space $L^{1,\infty}$, then it is bounded from $\LlL^{\alpha+1}$ to $\LlL^{\alpha}$ for all $\alpha\geq 0$ (see Theorem 4.4.11 and Theorem 4.6.16 in \cite{BS88}).
    
    Since $m$ is bounded, then $T_m$ is bounded from $L^2(\T)$ to itself. Now, we will conclude the proof by proving that $T_m$ is bounded from $L^1$ to $L^{1,\infty}$. 

    In the remaining part of the proof $|A|$ denotes the Lebesgue measure of any Borel set $A$. It can be easily proved that 
    \begin{multline*}
        T_m(f)(t)
          =\frac{1}{4\pi}\,\text{p.v.}\int_{-\pi}^{\pi}f(t-\theta e_j)\arctan\Big(\frac{\theta}{2}\Big)\,d\theta\\
          :=\frac{1}{4\pi}\lim_{\epsilon\to 0}\int_{\epsilon<|\theta|<\pi}f(t-\theta e_j)\arctan\Big(\frac{\theta}{2}\Big)\,d\theta,
    \end{multline*}
    where $e_j$ denotes the $j^{\text{th}}$ element of the canonical basis.
    
    By a density argument, we can assume that $f$ is a finite linear combination of characteristic functions of disjoint dyadic intervals in $[-\pi,\pi]^3$. Given $\alpha>0$, then by the Calderon-Zygmund decomposition (see \cite{G08}), there exist $g,b$ such that $\hat{f}=g+b$, where $\hat{f}$ is an extension of $f$ at $[-\pi,\pi]^3\cup([-\pi,\pi]^3+2\pi e_j)\cup ([-\pi,\pi]^3-2\pi e_j)$ by periodicity.
    In addition, the following properties hold,
    \begin{enumerate}
        \item $\|g\|_{L^1}\lesssim \|f\|_{L^1}$, $\|g\|_{L^{\infty}}\lesssim \alpha$,
        \item $b=\sum_jb_j$, where the $b_j$ functions are supported on pairwise disjoint dyadic cubes $Q_j$,
        \item $\int_{Q_j}b_j(x)dx=0,$ $\|b_j\|_{L^1}\lesssim \alpha |Q_j|$ and $\sum_j|Q_j|\lesssim\frac1\alpha\|f\|_{L^1}$,
        \item The particular structure of $f$ implies that there is a finite number of $b_j$'s functions.
    \end{enumerate}
    For all $j$, define
    \[
      \overline{Q_j}
        =\{x\in \R^3:\text{there is }\epsilon \in [-l(Q_j),l(Q_j)] \text{ s.t. } x+\epsilon e_j \in Q_j\},
    \]
    where $l(Q_j)=|Q_j|^{\frac{1}{3}}$ is the length of the side of $Q_j$. Now,
    \begin{equation*}
      \begin{aligned}
        \lefteqn{|\{x \in [-\pi,\pi]^3: |T_m(f)(x)|>\alpha\}|\leq}\qquad&\\
          &\leq |\{x \in [-\pi,\pi]^3: |T_m(g)(x)|>\tfrac12\alpha\}|
            + |\cup_j\overline{Q_j}|\\
          &\quad + |\{x \in [-\pi,\pi]^3\cap (\cup_j\overline{Q_j})^c: |T_m(b)(x)|>\tfrac12\alpha\}|\\
          &\lesssim \alpha^{-1}\|f\|_{L^1}
            + |\{x \in [-\pi,\pi]^3\cap (\cup_j\overline{Q_j})^c: |T_m(b)(x)|>\tfrac12\alpha\}|,
      \end{aligned}
    \end{equation*}
    where the second inequality follows from Markov's inequality and the properties satisfied by $b$ and $g$.

    The last term on the right-hand side can be estimated by Markov's inequality and 
    \begin{multline*}
      \int_{[-\pi,\pi]^3\cap (\cup_j\overline{Q_j})^c}|T_m(b)(x)|\,dx
        \le \sum_j \int_{[-\pi,\pi]^3\cap (\overline{Q_j})^c} |T_m(b_j)(x)|\,dx\\
        \le \sum_j \int_{[-\pi,\pi]^3\cap (\overline{Q_j})^c} \int_{l(Q_j)<|\theta|<\pi}\big|\arctan(\tfrac{\theta}{2})b_j(x-\theta e_j)\big|\,d\theta\,dx.
    \end{multline*}
    Now if $l(Q_j)\geq \frac{\pi}{4}$, then 
    \begin{equation*}
      \int_{[-\pi,\pi]^3\cap (\overline{Q_j})^c} \int_{l(Q_j)<|\theta|<\pi}\big|\arctan(\tfrac{\theta}{2})b_j(x-\theta e_j)\big|\,d\theta \,dx\lesssim \|b_j\|_{L^1}.
    \end{equation*}
    Otherwise, 
    \begin{equation*}
      \begin{split}
        \lefteqn{\int_{[-\pi,\pi]^3\cap (\overline{Q_j})^c} \int_{l(Q_j)<|\theta|<\pi}|\arctan(\tfrac{\theta}{2})b_j(x-\theta e_j)|\,d\theta\,dx\leq}\qquad&\\
        &=\int_{[-\pi,\pi]^3\cap (\overline{Q_j})^c}\Bigl( \int_{l(Q_j)<|\theta|<\frac{\pi}{4}+l(Q_j)}|\arctan(\tfrac{\theta}{2})b_j(x-\theta e_j)|\,d\theta\\
        &\quad + \int_{\frac{\pi}{4}+l(Q_j)<|\theta|<\pi}|\arctan(\tfrac{\theta}{2})b_j(x-\theta e_j)|\,d\theta\Bigr)\, dx\\
        &\lesssim \|b_j\|_{L^1}\Big(\int_{l(Q_j)}^{\frac{\pi}{4}+l(Q_j)}\arctan(\tfrac{\theta}{2})\,d\theta +\pi \arctan(\tfrac{\pi}{8})\Big)\\
        &\le\|b_j\|_{L^1}\log(\frac{\sin(\frac{1}{2}(\frac{\pi}{4}+l(Q_j)))}{\sin(\frac{1}{2}l(Q_j))})\\
        &\le C \|b_j\|_{L^1};
      \end{split}
    \end{equation*}
    since $x\mapsto\log(\frac{\sin(\frac{1}{2}(\frac{\pi}{4}+x))}{\sin(\frac{1}{2}x)})$ is bounded on $[0,\frac{\pi}{4}]$.
    In conclusion, by the properties of the $b_j$ functions, it follows that the remaining term is bounded from below by $\frac{C}{\alpha}|f|_{L^1}$.
\end{proof}

\begin{prop}\label{IterRiesz}
Given $\alpha \in \N^*$ and $\alpha_1,\alpha_2,\alpha_3 \in \N$ such that $\alpha_1+\alpha_2+\alpha_3=\alpha$, consider the iterated Riesz multiplier,
    \begin{equation*}
        m(k)=\mathbf{1}_{\{k\neq 0\}}(k)k_1^{\alpha_1}k_2^{\alpha_2}k_3^{\alpha_3}/|k|^{\alpha}.
    \end{equation*}    
    Then, $T_m$ is bounded from $\LlL^{\alpha+1}$ to $\LlL^{\alpha}$ for all $\alpha\geq 0$.
\end{prop}
\begin{proof}
    Similarly to the proof of the previous result, it is enough to show that $m$ is bounded and that $T_m$ is bounded from $L^1$ to $L^{1,\infty}$.

    While the first requirement is trivial, i.e. the boundedness of $m$, instead, the second one follows directly from Mihlin-Hormander theorem (see Remark 2.9 in \cite{RW15}).
\end{proof}

\section{Proof of the Iteration Lemma}\label{c:IterLemma}
This section is devoted to the proof of \cref{IterLemma}, namely we modify a solution $(v_{q},p_{q},R_{q})$ to \cref{ApproxEq} in order to obtain a new solution $(v_{q+1},p_{q+1},R_{q+1})$ with $\|R_{q+1}\|_{\LinfL}$  sufficiently small.

\subsection{Impulsed Beltrami flows}
In this section we build the impulsed Beltrami flows, a modification of the intermittent Beltrami flows used by Buckmaster and Vicol \cite{BV19}, which will have a crucial role in the construction of $v_{q+1}$.

Consider $\Lambda=\Lambda^+ \cup\Lambda^-\subset \Q^3$, where
\begin{equation*}
  \Lambda^+=\Bigl\{\frac{1}{5}(3e_1\pm4e_2),\frac{1}{5}(3e_2\pm4e_3),\frac{1}{5}(3e_3\pm4e_1)\Bigr\}.
\end{equation*}
and $\Lambda^-:=-\Lambda^+$. Here $e_1$, $e_2$, $e_3$ is the canonical basis of $\R^3$.
Then, by construction, $\Lambda$ is a finite symmetric subset of $\Q^3$, with vectors of unit length and such that 
\begin{equation*}
    5\Lambda \subset \Z^3,\qquad
    \min_{\xi, \xi' \in \Lambda, \xi+\xi'\neq 0} |\xi +\xi'|\geq \frac{1}{5}.
\end{equation*}
Moreover, the following geometric lemma holds:
\begin{prop}{$($\cite{BV19,BV20}$)$}\label{GeomLemma}
    Let $B_{1}(I_3)$ be the ball of symmetric $3\times3$ matrices centered at $I_3$ with radius $1$. Then, there exist $\epsilon_{\Lambda}>0$ and a family $(\gamma_{\xi})_{\xi \in \Lambda}\subset C^{\infty}(B_{\epsilon_{\Lambda}}(I_3))$ such that
    \begin{enumerate}
        \item $\gamma_{\xi}=\gamma_{-\xi}$;
        \item for each $R \in B_{\epsilon_{\Lambda}}(I_3)$, 
        \begin{equation}\label{GeomEquation}
            R=\sum\limits_{\xi}\frac{1}{2}\gamma_{\xi}(R)^2(I_3-\xi \otimes \xi).
        \end{equation}
    \end{enumerate}
\end{prop}
Now, for any $\xi \in \Lambda$, consider $\Axi \in \mathbb{Q}^3$ such that 
    \begin{equation*}
     \Axi\cdot \xi=0,\qquad  A_{-\xi}=\Axi, \qquad |\Axi|=1,  
    \end{equation*}        
and define $B_{\xi}=\frac{1}{\sqrt{2}}(A_{\xi}+i \xi \times A_{\xi})$.

Therefore for any  $(\am)_{\xi \in \Lambda}\subset \mathbb{C}$ such that  $a_{-\xi}=\overline{a_{\xi}}$, we can define $W_{\xi}=B_{\xi}e^{i\lambda x\cdot \xi}$ and consider the classical Beltrami flows $W=\sum_{\xi}a_{\xi}W_{\xi}$ introduced in \cite{BV19,BV20}, which have various interesting properties such as
\begin{equation*}
    \fint _{\T}W \otimes W dx =\sum_{\xi}\frac{|a_{\xi}|^2}{2}(I_3-\xi\otimes \xi).
\end{equation*}

The impulsed Beltrami flows are obtained by modifying the classical Beltrami flows with suitable spatial and temporal intermittency. In particular, given $r \in \N$, consider the Fejer kernel
\begin{equation*}
    F(y)=\frac{1}{r+1}\Bigg(\frac{\sin((r+1)\frac{y}{2})}{\sin \frac{y}{2}}\Bigg)^2=\sum_{j=-r}^r \Bigg(1-\frac{|j|}{r+1}\Bigg)e^{ijy},
\end{equation*}
for all $y \in \mathbf{T}$, and define the $3D$-Fejer kernel as
\begin{equation}\label{3DNormFejer}
    F_r(x)=(2\pi)^{\frac{3}{2}}\prod_{i=1}^3F(x_i).
\end{equation}
Let us summarise some properties of the $1D$-Fejer kernels, which can be easily extended to $3D$-Fejer kernels using Fubini-Tonelli theorem. We point out the crucial fact that, unlike the Dirichlet kernels, the $L^1$ estimate does not contain logarithmic corrections. 

\begin{prop}\label{FejerEstimates}
Let r $\in 2\N^*$ and $n \in \N$, then
    \begin{align*}
        \|F\|_{L^1}&=1, & \|F\|_{L^2}&\approx r^{1/2}, & \|F\|_{L^\infty}&\lesssim r,\\
        \|\tfrac{d^n}{dy^n}F\|_{L^1}&\lesssim r^n, & \|\tfrac{d^n}{dy^n}F\|_{L^2}&\lesssim r^{n+\frac{1}{2}}, & \|\tfrac{d^n}{dy^n}F\|_{L^{\infty}}&\lesssim r^{n+1}.
    \end{align*}
\end{prop}
\begin{proof}
  The proof of the first three inequalities can be found in \cite{G08}.
  Consider then the case $n>0$. The $L^2$ and $L^{\infty}$ norms can be easily estimated, respectively, by Fourier coefficients and by the formula
        \begin{equation*}
            \sup_{[-\pi,\pi]}\Big|\frac{d^n}{dy^n}F(y)\Big|\lesssim \sum_{j=1}^r(1-\frac{j}{r+1})j^n.
        \end{equation*}
Instead, the estimate of the $L^1$ norm follows from the equalities:
        \begin{gather*}
             \frac{d}{dy}F(y)=\frac{2}{r+1}\frac{\sin((r+1)\frac{y}{2})}{\sin \frac{y}{2}}\frac{d}{dy}\bg[\frac{\sin((r+1)\frac{y}{2})}{\sin \frac{y}{2}}\bg]=\frac{2}{r+1}D_{k}(y)\frac{d}{dy}D_{k}(y),\\
             \frac{d^n}{dy^n}F(y)=\frac{1}{r+1}\sum_{j=0}^{n-1} C(n,j)\Bigl(\frac{d^j}{dy^j}D_k(y)\Bigr)\Bigl(\frac{d^{n-j}}{dy^{n-j}}D_k(y)\Bigr),
        \end{gather*}
where $C(n,j)$ are constants depending on $n$ and $j$, $k=\frac{r}{2}$ and $D_k$ is the Dirichlet kernel, i.e.
        \begin{equation*}
            D_k(y)=\frac{\sin((k+\frac{1}{2})y)}{\sin(\frac{y}{2})}=\sum_{j=-k}^k e^{ijy}.
        \end{equation*}
        Indeed, the $L^1$ norm of the derivatives can be estimated by means of the previous equalities, H\"older inequality and since $\|\frac{d^n}{dy^n}D_k(y)\|_{L^2} \approx k^n k^{1/2}$.
\end{proof}
For any $\xi \in \Lambda$, its corresponding impulsed Beltrami flow is defined by
\begin{equation}\label{BBDef}
  \BW=\et W_{\xi} ,
\end{equation}
where for every $x\in\T$, $t\geq0$ and $\xi\in\Lambda^+$,
\begin{equation}\label{etDefn}
  \begin{aligned}
    M_r^{\xi}(x,t)&=F_r(\lambda \sigma(x\cdot \xi+\mu t,x\cdot \Axi, x \cdot \xiA)),\\
    M^{\xi}_r&=M^{-\xi}_r,\qquad\text{for every }\xi \in \Lambda^-,\\
    \et(x,t)&=\frac{1}{\|F\|_{L^2}^3} M_r^{\xi}(x,t).
  \end{aligned}
\end{equation}
The parameters $\lambda, \sigma, \mu$ and $r$ will be chosen later. In particular, they must satisfy the following conditions,
\begin{equation}\label{AssumPams}
  \begin{aligned}
    \lambda \sigma r&>1, & 
    \sigma r &<\frac{1}{2}, & 
    \sigma \lambda &\in \N,\\
    r^{\frac{3}{2}}< \mu&<\lambda^2, &
    \lambda &\in 10\N^*, &
    r &\in 2\N^*,
  \end{aligned}
\end{equation}
in order to control the $\LinfL$ norm of the stress tensor $R_{q+1}$ by means of spatial and temporal intermittency, see \cref{subs:stressEstimate}.

\begin{oss}\label{DiffBB}
  The main difference between our building block and the intermittent Beltrami waves presented in \cite{BV19} is that we introduce spatial intermittency by means of Fejer kernels instead of Dirichlet kernels. 

  This modification does not affect the most important properties of the building blocks, such as the representation formula or the active Fourier coefficients. At the same time, it allows to prove the non-uniqueness result assuming a weaker condition on the parameter $\beta$.
\end{oss}
\begin{oss}
The time intermittency has been introduced in order to estimate a specific term in $\Div(R_{q+1})$, which is related to
    \begin{equation*}
        \Div(\BW \otimes \textbf{W}_{-\xi} +\textbf{W}_{-\xi} \otimes \BW)=\nabla \et^2-((\xi \cdot \nabla)\et^2)\xi=\nabla \et^2-\frac{\xi}{\mu}\partial_t\et^2.
    \end{equation*}
\end{oss}

\begin{prop}\label{MainPropBB}
  Assume that the conditions in \cref{AssumPams} on the parameters hold, then:
  \begin{align}
    \label{Point1}
      &\proj_{\le 2\lambda }\et=\et,\qquad
      \proj_{[\frac{\lambda}{4},3\lambda)}\BW=\BW,\\
    \label{Point2}
      &\frac{1}{\mu}\partial_t \et =(\xi \cdot \nabla) \et,\qquad\text{for all }\xi \in \Lambda^+,\\
    \label{Point3}
      &\proj_{[\frac{\lambda}{10},4\lambda)} \BW \otimes\BWpr =\BW \otimes \BWpr,\qquad\text{for }\xi \neq -\xi'.
  \end{align}
  Moreover, for all $N,M\in \N$ and $p\in[1,\infty]$,
  \begin{align}
    \label{Point4}
      \|\nabla^N \partial_t^M \et\|_{L^{\infty}_tL^p_x}&\lesssim (\lambda \sigma r)^N (\lambda \sigma r \mu)^M r^{\frac{3}{2}-\frac{3}{p}},\\
    \label{Point5}
      \|\nabla^N \partial_t^M \BW\|_{L^{\infty}_t L^p_x}&\lesssim \lambda ^N (\lambda \sigma r \mu)^M r^{\frac{3}{2}-\frac{3}{p}}.
  \end{align}
  Finally, for all $N,M\in \N$ and $\alpha\geq0$,
  \begin{align}
    \label{Point6}
      \|\nabla^N \partial_t^M \et\|_{L^{\infty}_t\LlL^{\alpha}_x}&\lesssim (\lambda \sigma r)^N (\lambda \sigma r \mu)^M r^{-\frac{3}{2}}\log^{\alpha}(\lambda),\\
    \label{Point7}
      \|\nabla^N \partial_t^M \BW\|_{L^{\infty}_t \LlL^{\alpha}_x}&\lesssim \lambda ^N (\lambda \sigma r \mu)^M r^{-\frac{3}{2}}\log^{\alpha}(\lambda).
  \end{align}
  Here the Banach spaces $\LlL^{\alpha}$ are defined in \cref{DefnLlL}.
\end{prop}
\begin{proof}
  The equalities in \cref{Point1}, \cref{Point2}, \cref{Point3} are true by construction.
  The inequalities in \cref{Point4}, \cref{Point5} follow from Fubini-Tonelli theorem, \cref{FejerEstimates} and the fact that the map 
  \begin{equation*}
    x \mapsto (\lambda \sigma (x\cdot \xi +\mu t,\Axi \cdot x,\xiA \cdot x))
  \end{equation*}
  is the composition of three volume preserving maps of $\T$.

  To prove \cref{Point6}, recall that $\et\approx r^{-\frac{3}{2}}M^{\xi}_r$ (\cref{etDefn} and \cref{FejerEstimates}), so 
    \begin{equation}\label{etAndM}
     \|\nabla^N \partial_t^M \et\|_{L^{\infty}_t\LlL^{\alpha}_x}\approx r^{-\frac{3}{2}}\|\nabla^N \partial_t^M M^{\xi}_r\|_{L^{\infty}_t\LlL^{\alpha}_x}.   
    \end{equation}  
    Now, let us define
    \[
      C:=\int_{\T}|\nabla^N \partial_t^M M^{\xi}_r|\log^{\alpha}(2+|\nabla^N \partial_t^M M^{\xi}_r|)\,dx
    \]
    and distinguish two cases:
    \begin{itemize}
        \item If $C\le 1$, then \cref{LuxDefn} implies that
        \begin{equation*}
            \|\nabla^N \partial_t^M M^{\xi}_r\|_{L^{\infty}_t\LlL^{\alpha}_x}\le 1\lesssim (\lambda \sigma r)^N (\lambda \sigma r \mu)^M \log^{\alpha}(\lambda),
        \end{equation*}
        therefore, 
        \begin{equation*}
            \|\nabla^N \partial_t^M \et\|_{L^{\infty}_t\LlL^{\alpha}_x}\lesssim(\lambda \sigma r)^N (\lambda \sigma r \mu)^M r^{-\frac{3}{2}} \log^{\alpha}(\lambda).
        \end{equation*}
        \item If $C>1$, since $A(s)=s\log^{\alpha}(2+s)$ is a convex function and $A(0)=0$, then, by the definition of Luxemburg norm, we have that
        \begin{equation*}
           \|\nabla^N \partial_t^M M^{\xi}_r\|_{L^{\infty}_t\LlL^{\alpha}_x}\lesssim \|\nabla^N \partial_t^M M^{\xi}_r\|_{L^{\infty}_tL^1_x}\log^{\alpha}(\|\nabla^N \partial_t^M M^{\xi}_r\|_{L^{\infty}_tL^{\infty}_x}),
        \end{equation*}
        then,
        \begin{align*}
          \|\nabla^N \partial_t^M \et\|_{L^{\infty}_t\LlL^{\alpha}_x}&\lesssim r^{-\frac{3}{2}}\|\nabla^N \partial_t^M M^{\xi}_r\|_{L^{\infty}_tL^1_x},\\
           \log^{\alpha}(r^{\frac{3}{2}}\|\nabla^N \partial_t^M \et\|_{L^{\infty}_tL^{\infty}_x})&\lesssim(\lambda \sigma r)^N (\lambda \sigma r \mu)^M r^{-\frac{3}{2}} \log^{\alpha}(\lambda),
        \end{align*}
        where the last inequality is due to \cref{etAndM} and to \cref{Point4}.
    \end{itemize}

    We finally turn to the proof of \cref{Point7}. \cref{Point6} and \cref{PropLuxemburg} imply that
    \begin{multline*}
      \|\nabla^N \partial_t^M \BW\|_{L^{\infty}_t \LlL^{\alpha}_x}
        \lesssim\sum_{k=0}^N\|\nabla^k \partial_t^M\et\|_{L^{\infty}_t \LlL^{\alpha}_x}\|\nabla ^{N-k}W_\xi\|_{L^{\infty}_t L^{\infty}_x}
        \lesssim\\
        \lesssim\sum_{k=0}^N (\lambda \sigma r)^k (\lambda \sigma r \mu)^M r^{-\frac{3}{2}} \log^{\alpha}(\lambda) \lambda^{N-k}
        \le \lambda^N (\lambda \sigma r \mu)^M r^{-\frac{3}{2}} \log^{\alpha}(\lambda),
    \end{multline*}
    which concludes the proof.
\end{proof}

\begin{prop}\label{RepresFormula}
    Consider the family of vector fields $(\BW)_{\xi \in \Lambda}$ and a family of complex values $(a_{\xi})_{\xi \in \Lambda}$ such that $\overline{a_{-\xi}}=a_{\xi}$. Then, $\sum_{\xi \in \Lambda}a_{\xi}\BW$ is a real-valued vector field and for all $R \in B_{\epsilon_{\Lambda}}(I_d)$,
    \begin{equation*}
        \sum_{\xi \in \Lambda}\gammxi^2\fint_{\T}\BW \otimes \mathbf{W}_{-\xi}dx=R, 
    \end{equation*}
    where $\epsilon$ and the functions $\gammxi$ are given in \cref{GeomLemma}.
\end{prop}
\begin{proof}
  We have that
  \begin{multline*}
    \sum_{\xi \in \Lambda}\gammxi^2\fint_{\T}\BW \otimes \textbf{W}_{-\xi}dx
      =\sum_{\xi \in \Lambda}\gammxi^2B_{\xi} \otimes B_{-\xi}\fint_{\T}\et(x,t)^2dx=\\
      =\sum_{\xi \in \Lambda}\gammxi^2B_{\xi} \otimes B_{-\xi}
      = \sum_{\xi \in \Lambda^+}\gammxi^2(I_3-\xi \otimes \xi)=\\
      =\frac{1}{2}\sum_{\xi \in \Lambda}\gammxi^2(I_3-\xi \otimes \xi)
      =R,
  \end{multline*}
  where the second equality follows from the normalization introduced in \cref{3DNormFejer} and \cref{etDefn}, the fifth one follows from the symmetry of $\Lambda$ and \cref{GeomLemma}, and the last one follows from \cref{GeomEquation}.
\end{proof}

\subsection{The perturbation}
Following the idea proposed by Luo et al. \cite{LT20}, consider a cut-off function $\psi$ that controls the time support of the perturbation, i.e. a function that has the following properties
\begin{gather*}
  \psi(t)=1 \text{ on } \supp_t(R_q),\quad
  \supp_t(\psi)\subset N_{\delta_{q+1}}(\supp_t(R_q)),\quad
  \text{ and }\quad
  |\psi'(t)|\le 2\delta_{q+1}^{-1}.
\end{gather*}
Given a smooth non-decreasing function such that
\begin{equation*}
  \chi(s)=
    \begin{cases}
      1, &\qquad 0\le s\le 1,\\
      s, &\qquad s\geq 2,
    \end{cases}
\end{equation*}
define a function that smoothly recovers the magnitude of the tensor stress $R_q$,
\begin{equation}\label{eq:rho}
    \rho(t,x):=\frac{2}{\epsilon_{\Lambda}}\chi(\delta_{q+1}^{-1}|R_q(t,x)|)\delta_{q+1}\psi^2(t),
\end{equation}
where $\epsilon_{\Lambda}$ is the value from \cref{GeomLemma}.

Now, given $\xi \in \Lambda$, we define the magnitude of its impulsed Beltrami flow, in order to recover $R_q$ through \cref{RepresFormula}, as
\begin{equation}\label{eq:axi}
    a_{\xi}(x,t)=\rho(x,t)^{\frac{1}{2}}\gamma_{\xi}\bigg(I_3-\frac{R_q(x,t)}{\rho(x,t)}\bigg).
\end{equation}
Indeed, since
\begin{equation*}
    \frac{|R_q|}{\rho}\le \frac{\epsilon_{\Lambda}}{2}\qquad\text{on } \supp_t(R_q),
\end{equation*}
then
\begin{equation}\label{RecoverRq}
    \sum_{\xi \in \Lambda}a_{\xi}^2\fint\BW \otimes \mathbf{W}_{-\xi}dx=\rho I_3-R_q.
\end{equation}
We can finally define the new velocity field as $v_{q+1}=v_q+w$, where
\begin{equation}\label{eq:w}
  \begin{aligned}
    w&=w^{(p)}+w^{(c)}+w^{(t)},\\
    w^{(p)}&=\sum_{\xi \in \Lambda}a_{\xi}\BW,\\
    w^{(c)}&=\frac{1}{\lambda}\sum_{\xi \in \Lambda}\nabla(a_{\xi}\eta_{\xi})\times W_{\xi},\\
    w^{(t)}&=\frac{1}{\mu}\sum_{\xi \in \Lambda^+}P_H\proj_{\neq 0}(a_{\xi}^2\et^2 \xi).
  \end{aligned}
\end{equation}
Notice that the function $\psi$ allows to identify the temporal support of the perturbation $w$. Indeed,
\begin{equation}\label{TempSupp}
    \supp_t(w)\subset \supp_t(\rho) \subset \supp_t(\psi)\subset N_{\delta_{q+1}}(\supp_t R_{q+1}).
\end{equation}
Now, let us recall various estimates regarding the perturbation $w$.

\begin{lemma}[Lemma 2, \cite{LT20}]\label{Lemma1}
  Let $\rho$, $a_\xi$ be defined as in \cref{eq:rho}, \cref{eq:axi}, then
  \begin{gather*}
      \|\rho\|_{\LinfL}\le C\delta_{q+1},\qquad
      \|\rho\|_{C^0(\supp_t(R_q))}\lesssim\delta_{q+1}^{-1},\qquad
      \|\rho\|_{C^N_{t,x}}\lesssim C(\delta_{q+1},|R_q|_{C^N_{t,x}}),\\
      \|a_{\xi}\|_{\LinfLSec}\lesssim \delta_{q+1}^{\frac{1}{2}},\qquad
      \|a_{\xi}\|_{C^N_{t,x}}\lesssim C(\delta_{q+1},|R_q|_{C^N_{t,x}}).
  \end{gather*}
\end{lemma}

\begin{lemma}[Lemma 2.1, \cite{MS18}]\label{Lemma2}
  Given $1\le p \le \infty$, if $f,g\in C^{\infty}(\T)$, with $g$ $(\frac{\mathbf{T}}{N})^3$-periodic and $N \in \N$, then,
    \begin{equation*}
        \|fg\|_{L^p}\lesssim \|f\|_{L^p}\|g\|_{L^p}+N^{-\frac{1}{p}}\|f\|_{C^1}\|g\|_{L^p}.
    \end{equation*}
\end{lemma}

In the rest of the section we will denote by $C_a(N)$ any value that depends polynomially on $\sup_{n\le N}\|a_{\xi}\|_{C^n_{t,x}}$.

\begin{lemma}\label{MainEstimate}
  If $w$, $w^{(p)}$, $w^{(c)}$, $w^{(t)}$ are defined as in \cref{eq:w}, then
  \begin{align*}
    \|w^{(p)}\|_{\LinfLSec}&\lesssim \delta_{q+1}^{\frac{1}{2}}+C_a(1)(\lambda \sigma)^{-\frac{1}{2}},\\
    \|w^{(c)}\|_{L^{\infty}_tL^2_x}&\lesssim C_a(1)\Bigl(\frac{1}{\lambda}+\sigma r\Bigr),\\
    \|w^{(t)}\|_{L^{\infty}_tL^2_x}&\lesssim C_a(0)\frac{r^{\frac{3}{2}}}{\mu},\\
    \|w\|_{L^{\infty}_{t}L^{s}_x}&\lesssim C_a(0)r^{\frac{3}{2}-\frac{3}{s}},\qquad \text{ for } s \in (1,2]. 
  \end{align*}
\end{lemma}
\begin{proof}
  Since $\BW$ is $(\frac{\mathbf{T}}{\lambda \sigma})^3$-periodic (this follows from \cref{Point1}), \cref{Lemma2} can be applied to any term $a_{\xi}\BW$. Then, the first inequality follows from the estimates shown in \cref{Lemma1}, and from \cref{Point5}.

  The other inequalities from \cref{MainPropBB} and the fact that the Leray-Helmholtz projector $P_H$ is bounded from $L^q$ to itself for $1<q<\infty$.
\end{proof}

\subsection{Estimating the stress}\label{subs:stressEstimate}
In the previous section we have defined the perturbation $w$ and the new velocity field $v_{q+1}$. Now, we complete the proof of \cref{IterLemma} showing how to properly build the pressure $p_{q+1}$ and the tensor stress $R_{q+1}$.

Let us first introduce an anti-divergence operator, namely a right-inverse of the divergence operator, which maps vector fields into trace-free matrix-valued functions.

\begin{defn}\label{AntiDivDefn}
    Let $u\in C^{\infty}(\T,\R^3)$ be a smooth vector field. A \emph{anti-divergence} is the linear operator defined by
    \begin{equation*}
        \AntiDiv(u)=\frac{1}{4}\big(\nabla P_Hv+(\nabla P_Hv)^T)+\frac{3}{4}(\nabla v+(\nabla v)^T\big)-\frac{1}{2}\Div( v)I_d,
    \end{equation*}
    where $v\in C^{\infty}(\T,\R^3)$ is the unique solution to the problem
    \[
      \begin{cases}
        \Delta v=u-\fint u\,dx,\\
        \fint v\,dx=0,
      \end{cases}
    \]
    with periodic boundary conditions.
\end{defn}

The next proposition fully justifies the name given to the operator $\AntiDiv$.

\begin{prop}
    Consider $u\in C^{\infty}(\T,\R^3)$, then,
        \begin{equation}\label{AntiDiv}
            \Div(\AntiDiv(u))=u-\fint_{\T} u, \qquad \Tr(\mathcal{R}(u))=0.
        \end{equation} 
    Moreover, for all $k \in\Z^3$, with $k\neq0$,
    \begin{equation}\label{AntiDivFourier}
      \begin{aligned}
        [\mathcal{F}(\AntiDiv(u))(k)_j^h]_{j,h=1}^3
          & = \frac{i}{2|k|^2} (\mathcal{F}(u)(k)\cdot k)\mathbf{1}_{j=h}
            + \frac{i\,k_h k_j}{2|k|^4}(\mathcal{F}(u)(k)\cdot k) + {}\\
          & \quad - \frac{i k_h}{|k|^2}\mathcal{F}(u)(k)_j - \frac{i k_j}{|k|^2} \mathcal{F}(u)(k)_h, 
    \end{aligned}
    \end{equation}
    where $\mathcal{F}(v)(k)_j$ is the $j^\text{th}$ component of the $k^\text{th}$ Fourier coefficient of $v$.
\end{prop}
\begin{proof}
    This result follows from straightforward computations.
\end{proof}
\begin{oss}
    \cref{AntiDivFourier} shows that $|\nabla|\AntiDiv(u)$ is obtained by applying three or one Riesz transforms to the components of the vector field $u$.
\end{oss}

Since $(v_q,p_q,R_q)$ solves \cref{ApproxEq}, then
\begin{equation*}
  \begin{aligned}
    \lefteqn{\partial_t v_{q+1}+\Div(v_{q+1}\otimes v_{q+1})+\Diss v_{q+1}=}\qquad&\\
      &=\underbrace{\Div(R_q)-\nabla p_q+\Div(w^{(p)} \otimes w^{(p)})+\partial_tw^{(t)}}_\text{oscillation error} + {}\\
      &\quad + \underbrace{\Div(v_q\otimes w+w\otimes v_q)+\Diss w+\partial_t(w^{(c)}+w^{(p)})}_\text{linear error} + {}\\
      &\quad + \underbrace{\Div(w^{(p)} \otimes (w^{(c)}+w^{(t)})+(w^{(c)}+w^{(t)})\otimes w)}_\text{corrector error}. 
  \end{aligned}
\end{equation*}

Now, we show that the terms on the RHS can be expressed as the gradient of a scalar function $p_{q+1}$ and as the divergence of a matrix-valued function $R_{q+1}$. In addition, we shall prove that $\|R_{q+1}\|_{\LinfL}\le \delta_{q+2}$.

Before estimating $\|R_{q+1}\|_{\LinfL}$, let us prove a useful lemma.
\begin{lemma}\label{MultiplierAndPeriodicity}
 Consider $a,b \in C^{\infty}(\T)$ such that $\proj_{\geq N}b=b$.
\begin{enumerate}
    \item Let $\beta>3$ and $N \in 2\N^*$. Then,
    \begin{equation}\label{MultiplierAndPeriodicity1}
        \|\mathcal{R}|\nabla|T_M (ab)\|_{L^1}\le \log^{3-\beta}(N)\|a\|_{H^3}\|b\|_{\LlL^1},
    \end{equation}
    where $|\nabla|$ is the operator associated with the multiplier $m(k)=|k|$.
    \item Moreover, 
    \begin{equation}\label{MultiplierAndPeriodicity2}
        \|\mathcal{R}\proj_{\neq 0}(ab)\|_{L^1}\lesssim \frac1N\log^3(N)\|a\|_{H^4}\|b\|_{\LlL^2},
    \end{equation}
    where the space $\LlL^{\alpha}$ is defined in \cref{DefnLlL}.
\end{enumerate}  
\end{lemma}
\begin{proof}
  Define $a_1=\proj_{<\frac{N}{2}}(a)$ and $a_2=\proj_{\geq \frac{N}{2}}(a)$. The assumption on $b$ implies that
\begin{equation*}
    ab=\proj_{\geq\frac{N}{2}}(a_1b)+(a_2b).
\end{equation*}
Notice that \cref{AntiDivFourier} implies that the operator $\mathcal{R}|\nabla|T_M$ corresponds to the multiplier
\begin{equation*}
\frac{m_1(k)}{|k|}M(k)|k|,
\end{equation*}
where $m_1$ is a suitable linear combination of iterated Riesz multipliers (see \cref{IterRiesz}). Now,
\begin{multline*}
  \|\mathcal{R}|\nabla|T_M (ab)\|_{L^1}
    \lesssim \log^{3-\beta}(N)\|a_1b\|_{\LlL^1}+\|a_2b\|_{\LlL^1}\le\\
    \leq\log^{3-\beta}(N)\|a_1\|_{L^{\infty}}\|b\|_{\LlL^1}+\|a_2\|_{L^{\infty}}\|b\|_{\LlL^1}\le\\
     \leq\bigl(\log^{3-\beta}(N) \|a\|_{H^2}+\tfrac{1}{N}\|a\|_{H^3}\bigr)\|b\|_{\LlL^1},
\end{multline*}
where the first inequality is due to  \cref{IterRiesz} and \cref{ProjEstimateDiv-1log}, the second one follows from \cref{PropLuxemburg} and the third one can be justified by Sobolev embeddings.

The second inequality, \cref{MultiplierAndPeriodicity2}, follows from a slight modification of the previous proof. In particular, recall that $\mathcal{R}\proj_{\neq 0}$ can be seen as corresponding to the Fourier multiplier given by \cref{AntiDivFourier}, which is equal to
\begin{equation*}
    \frac{m_1(k)}{|k|_1}\frac{\sum_{j=1}^3 |k_j|}{|k|}=\frac{m_1(k)}{|k|_1}\frac{\sum_{j=1}^3 \sign(k_j)k_j}{|k|}=\frac{1}{|k|_1}\sum_j m_{2,j}(k) \sign(k_j),
\end{equation*}
where $m_1$ and $m_{2,j}$ are linear combinations of iterated Riesz transforms.
Therefore, $\mathcal{R}\proj_{\neq 0}$ is a finite linear combination of operators given by the composition of $T_m$ with $m(k)=\mathbf{1}_{\{k\neq 0\}}(k)\frac{1}{|k|_1}$, iterated Riesz transforms and directional Hilbert transforms. Then, the result is easily achieved by using \cref{Div-1}, \cref{IterRiesz} and \cref{DirHilbert}.
\end{proof}

\subsubsection{Corrector error}

 Given its formulation, it produces two addends of $R_{q+1}$ whom $\LinfL$ norm can be estimated thanks to \cref{MainEstimate}. Indeed,
\begin{align*}
  \lefteqn{\|w^{(p)} \otimes (w^{(c)}+w^{(t)})+(w^{(c)}+w^{(t)})\otimes w\|_{\LinfL}\le}\qquad&\\
    &\leq\|w+w^{(p)}\|_{\LinfLSec}+\|w^{(c)}+w^{(t)}\|_{\LinfLSec}\\
    &\le C_a(1) \Bigl(\frac{1}{\lambda}+\sigma r + \frac{r^{\frac{3}{2}}}{\mu}\Bigr)\Bigl(\frac{1}{\lambda}+\sigma r + \frac{r^{\frac{3}{2}}}{\mu}+\delta_{q+1}^{\frac{1}{2}}+(\lambda \sigma)^{-\frac{1}{2}}\Bigr).
\end{align*}

\subsubsection{Linear error}

 Since $\mathcal{R}$ is the right inverse to the divergence operator (\cref{AntiDiv}), the terms $v_q \otimes w+w\otimes v_q$ and $\mathcal{R} (\Diss w+\partial_t(w^{(c)}+w^{(p)}))$ are part of $R_{q+1}$. In addition, \cref{MainEstimate} implies that
 \begin{equation*}
   \|v_q\otimes w\|_{\LinfL}\le \|v_q\|_{L^{\infty}_tL^{\infty}_x} \|w\|_{L^{\infty}_t L^s_x} \lesssim C_a(0) \|v_q\|_{L^{\infty}_tL^{\infty}_x} r^{\frac{3}{2}-\frac{3}{s}},
 \end{equation*}
 for $s\in (1,2)$.

The remaining terms are estimated through \cref{MultiplierAndPeriodicity}.
\begin{prop}\label{linearEst}
    The following estimates hold:
    \begin{enumerate}
        \item $\|\mathcal{R}(\partial_t (w^{(c)}+w^{(p)}))\|_{\LinfL}\lesssim C_a(1) \sigma \mu r^{-\frac{1}{2}}\log(\lambda)$,
        \item $\displaystyle\|\mathcal{R}\Diss(w^{(p)}+w^{(c)})\|_{\LinfL}\lesssim C_a(2)\Bigl(\frac\lambda{r}\Bigr)^{\frac{3}{2}}\log^{4-\beta}(\lambda)$;
        \item $\displaystyle\|\mathcal{R}\Diss(w^{(t)})\|_{\LinfL}\lesssim C_a(2)\frac{(\lambda \sigma r)^{\frac{3}{2}}}{\mu}\log(\lambda)$.
    \end{enumerate}
\end{prop}
\begin{proof}
  To prove the first inequality, notice that $w^{(p)}+w^{(c)}=\frac1\lambda\nabla \times w^{(p)}$, so
  \begin{multline*}
    \|\mathcal{R}(\partial_t (w^{(c)}+w^{(p)}))\|_{\LinfL}=\frac{1}{\lambda}\|\mathcal{R}(\nabla \times (\partial_t w^{(p)}))\|_{\LinfL}\lesssim\\
      \lesssim\frac{1}{\lambda}\|\partial_tw^{(p)}\|_{L^{\infty}_t\LlL^1_x}\le \frac{1}{\lambda}(C_a(1)\|\et\|_{L^{\infty}_t\LlL^1_x} +C_a(0) \|\partial_t \et\|_{L^{\infty}_t\LlL^1_x}),
  \end{multline*}
where the first inequality is due to \cref{AntiDivFourier} and \cref{IterRiesz} and the second one follows from \cref{PropLuxemburg}.
Now, \cref{Point6} implies the required inequality.

We turn to the proof of the second inequality. Recall that $T_M$ is the operator associated with the multiplier \begin{equation*}
    M(k)=\mathbf{1}_{\{k\neq 0\}}(k)\log^{-\beta}(|k|_1+10).
\end{equation*} Then by a classical interpolation argument, we have
\begin{multline*}
  \|\mathcal{R}\Diss(w^{(p)}+w^{(c)})\|_{\LinfL}
    =\frac{1}{\lambda}\||\nabla|^{\frac{3}{2}}(\mathcal{R}|\nabla|T_M (\nabla \times w^{(p)}))\|_{\LinfL}\leq\\
    \le \frac{1}{\lambda}\|\mathcal{R}|\nabla|T_M \nabla(\nabla \times w^{(p)}))\|_{\LinfL}^{\frac{1}{2}}
      \|\mathcal{R}|\nabla|T_M \nabla^2(\nabla \times w^{(p)}))\|_{\LinfL}^{\frac{1}{2}}.
\end{multline*}
Now we estimate the first factor $\|\mathcal{R}|\nabla|T_M\nabla(\nabla \times w^{(p)}))\|_{\LinfL}$, the second one can be estimated by a similar procedure.

First of all,
\begin{multline*}
  \nabla(\nabla \times (a_{\xi}\BW))=\nabla(a_{\xi}(\nabla \times \BW)+\nabla a_{\xi} \times \BW)=\\
  = (\nabla a_{\xi})(\nabla \times \BW)+a_{\xi}\nabla(\nabla \times \BW)+\nabla^2a_{\xi}\times \BW+\nabla a_{\xi}\times \nabla \BW.
\end{multline*}
Notice that \cref{Point1} implies that all the terms are composed by a suitable derivative of $a_{\xi}$ and a $(\frac{\mathbf{T}}{\lambda /5})^3$-periodic function. Therefore,
\begin{align*}
  \|\mathcal{R}|\nabla|T_M \nabla(\nabla \times w^{(p)}))\|_{\LinfL}
    &\lesssim \log^{3-\beta}(\lambda)C_a(5)\|\nabla ^2\BW\|_{L^{\infty}_t\LlL^1_x}\\
    &\lesssim C_a(5)\lambda^2r^{-\frac{3}{2}}\log^{4-\beta}(\lambda),
\end{align*}
where the first inequality is due to \cref{MultiplierAndPeriodicity1}, while \cref{Point7} implies the second one.

Finally, we prove the third inequality. Following the same strategy used above, we have that
\begin{align*}
  \|\mathcal{R}\Diss(w^{(t)})\|_{\LinfL}
    &=\||\nabla|^{\frac{3}{2}}(\mathcal{R}|\nabla|T_M (w^{(t)}))\|_{\LinfL}\\
    &\le \|\mathcal{R}|\nabla|T_M \nabla(w^{(t)}))\|_{\LinfL}^{\frac{1}{2}}
      \|\mathcal{R}|\nabla|T_M \nabla^2(w^{(t)}))\|_{\LinfL}^{\frac{1}{2}}.
\end{align*}
Let us show how we can estimate the first factor. Recall that the Leray-Helmholtz projector $P_H$ corresponds to the multiplier $I-\frac{k\otimes k}{|k|^2}$, therefore $\mathcal{R}|\nabla|P_H$ is associated with a linear combination of iterated Riesz transforms. So, by \cref{IterRiesz} and \cref{Div-1Log}, we have that $\mathcal{R}|\nabla|P_HT_M$ maps $\LlL$ into $L^1$. Hence, 
\begin{align*}
  \|\mathcal{R}|\nabla|T_M \nabla(w^{(t)}))\|_{\LinfL}
    &\lesssim\frac{1}{\mu}\|\mathcal{R}|\nabla| P_HT_M(\nabla(a_{\xi}^2)\et^2+a^2_{\xi}\nabla(\et^2))\|_{\LinfL}\\
    &\lesssim\frac{1}{\mu}\|\nabla(a_{\xi}^2)\et^2+a^2_{\xi}\nabla(\et^2)\|_{L^{\infty}_t\LlL^1_x}\\
    &\lesssim\frac{1}{\mu}C_a(1)\bigl(\|\et^2\|_{L^{\infty}_t\LlL^1_x}+\|\nabla\et^2\|_{L^{\infty}_t\LlL^1_x}\bigr)\\
    &\lesssim\frac{1}{\mu}C_a(1)\lambda \sigma r \log(\lambda),
\end{align*}
where the third inequality follows from \cref{PropLuxemburg} and the last one follows from \cref{Point6}.
\end{proof}

\subsubsection{Oscillation error}

Finally, the oscillation error is treated as in \cite{BV19}. So,
\begin{align*}
  \lefteqn{\Div(R_q-w^{(p)}\otimes w^{(p)})=}\qquad&\\
    &=\Div\bigl(R_q-\sum_{\xi,\xi'}a_{\xi}a_{\xi'}\fint_{\T}\BW\otimes\BWpr dx
      +\sum_{\xi,\xi'}a_{\xi}a_{\xi'}\proj_{\neq 0}\BW\otimes \BWpr\bigr)\\
    &=\Div\bigl(R_q-\sum_{\xi}a_{\xi}^2\fint_{\T}\BW\otimes \mathbf{W}_{-\xi}dx
      +\sum_{\xi,\xi'}a_{\xi}a_{\xi'}\proj_{\neq 0}\BW\otimes \BWpr\bigr)\\
    &=\Div\bigl(\rho I_3+\sum_{\xi,\xi'}a_{\xi}a_{\xi'}\proj_{\neq 0}\BW\otimes \BWpr\bigr),
\end{align*}
where the second equality follows from \cref{Point3} and the third one from \cref{RecoverRq}. Since $\Div(\rho I_3)=\nabla \rho$, then $\rho$ goes into $ p_{q+1}$ as well as $ p_q$.

Now, we can treat $\Div(\sum_{\xi,\xi'}a_{\xi}a_{\xi'}P_{\neq 0}\BW\otimes \BWpr)=:\sum _{\xi,\xi'}E_{\xi,\xi'}$ by taking into account each term $E_{\xi,\xi'}+E_{\xi',\xi}$. In particular, since $E_{\xi,\xi'}+E_{\xi',\xi}$ has zero mean, then 
\begin{align*}
  E_{\xi,\xi'}+E_{\xi',\xi}
    &=\proj_{\neq 0}\bigl(\nabla(a_{\xi}a_{\xi'})\cdot\proj_{\geq\lambda\sigma/2}(\BW\otimes\BWpr+\BWpr\otimes\BW)\bigr)\\
    &\quad + \proj_{\neq 0}\big(a_{\xi}a_{\xi'} \Div(\BW\otimes \BWpr +\BWpr \otimes \BW)\big)\\
    &=:E_{\xi,\xi',1}+E_{\xi,\xi',2},
\end{align*}
The term $\mathcal{R}E_{\xi,\xi',1}$ goes into $R_{q+1}$, \cref{MultiplierAndPeriodicity2} and \cref{PropLuxemburg} imply that 
\begin{multline*}
  \|\mathcal{R}E_{\xi,\xi',1}\|_{\LinfL}
    \lesssim \frac{\log^3(\lambda \sigma)}{\lambda\sigma }C_a(5)\|\BW\otimes \BWpr\|_{L^{\infty}_t\LlL^2_x}\lesssim\\
    \lesssim \frac{\log^3(\lambda \sigma)}{\lambda \sigma }C_a(5)\|\et \eta_{\xi'}\|_{L^{\infty}_t\LlL^2_x}
    \lesssim C_a(5) \frac{\log^3(\lambda \sigma)}{\lambda \sigma }  \log^2(\lambda),
\end{multline*}
where $\|\et \eta_{\xi'}\|_{L^{\infty}_t\LlL^2_x}$ can be estimated arguing as in \cref{Point6}.

Instead, we need to distinguish two cases in order to deal with $E_{\xi,\xi',2}$.

\underline{\emph{First case: $\xi\neq -\xi'$}}.
If $\xi\neq -\xi'$, then $\proj_{>\lambda/10}(\BW\otimes \BWpr)=\BW\otimes \BWpr$. Therefore,
    \begin{equation*}
      \begin{split}
        \lefteqn{a_{\xi}a_{\xi'} \Div((\BW\otimes \BWpr +\BWpr \otimes \BW)=}\qquad&\\
          &= a_{\xi}a_{\xi'} \Div( \proj_{\geq\lambda/10}\big(\et \eta_{\xi'}(W_{\xi}\otimes W_{\xi'}+W_{\xi'}\otimes W_{\xi})\big)\\
          &= a_{\xi}a_{\xi'} \proj_{\geq\lambda/10}\big(\nabla(\et \eta_{\xi'})\cdot (W_{\xi}\otimes W_{\xi'} + W_{\xi'}\otimes W_{\xi})\big) + {}\\
          &\quad +a_{\xi}a_{\xi'}\proj_{\geq\lambda/10}\big(\et \eta_{\xi'}\nabla(W_{\xi}\cdot W_{\xi'})\big)\\
          &=a_{\xi}a_{\xi'} \proj_{\geq\lambda/10}\big(\nabla(\et \eta_{\xi'})\cdot (W_{\xi}\otimes W_{\xi'}+W_{\xi'}\otimes W_{\xi})\big) +\nabla(a_{\xi}a_{\xi'}\BW\cdot\BWpr) + {}\\
          &\quad -\nabla(a_{\xi}a_{\xi'})\proj_{\geq\lambda/10}(\BW\cdot \BWpr)-a_{\xi}a_{\xi'}\proj_{\geq\lambda/10}\big(W_{\xi}\cdot W_{\xi'}\nabla (\et \eta_{\xi'})\big),
      \end{split}
    \end{equation*}
where the second term produces a pressure, the third one can be estimated as $E_{\xi,\xi',1}$, the first and the fourth produce two new terms of $R_{q+1}$ whose $\LinfL$ norms can be estimated in the same way. Namely,
\begin{multline*}
  \|\mathcal{R}\proj_{\neq 0}\Big(   a_{\xi}a_{\xi'}\proj_{\geq\lambda/10}\big(W_{\xi}\cdot W_{\xi'}\nabla (\et \eta_{\xi'})\big)\Big)\|_{\LinfL}\lesssim\\
  \lesssim C_a(5)\frac{\log^3(\lambda)}{\lambda}\|\nabla (\et \eta_{\xi'})\|_{L^{\infty}_t\LlL^2_x} \lesssim C_a(5)\frac{\log^3(\lambda)}{\lambda}\|\nabla (\et) \eta_{\xi'}\|_{L^{\infty}_t\LlL^2_x}\lesssim\\
  \lesssim C_a(5)\frac{\log^3(\lambda)}{\lambda} \log^2(\lambda) \lambda \sigma r,
\end{multline*}
where we have used \cref{MultiplierAndPeriodicity2}.

\underline{\emph{Second case: $\xi= -\xi'$}}.
If $\xi'=-\xi$, then 
$E_{\xi,-\xi,2}=\proj_{\neq 0}\big(a_{\xi}^2(\nabla\proj_{\geq \lambda\sigma/2}( \et^2)-\frac{\xi}{\mu}\partial_t \et^2)\big)$. Now, let us recall that $\partial_t w^{(t)}$ is part of the oscillation term, i.e. we need to deal with $\sum_{\xi \in \Lambda^+}[\proj_{\neq 0}\big(a_{\xi}^2(\nabla\proj_{\geq \lambda\sigma/2}( \et^2)-\frac{\xi}{\mu}\partial_t \et^2)\big)+\frac{1}{\mu}\partial_t P_H\proj_{\neq 0}(a_{\xi}^2 \et^2 \xi)]$.\\
Now,
\begin{equation*}
  \proj_{\neq 0}\big(a_{\xi}^2(\nabla\proj_{\geq \lambda\sigma/2}( \et^2))\big)=\nabla(a_{\xi}^2  \proj_{\geq \lambda\sigma/2}( \et^2)) - \proj_{\neq 0}(\nabla a_{\xi}^2\proj_{\geq \lambda\sigma/2}( \et^2)),
\end{equation*}
where the first term goes into the pressure $p_{q+1}$, while the second produces a term of $R_{q+1}$ with $\LinfL$ estimated using \cref{MultiplierAndPeriodicity2}. Indeed, 
\begin{equation*}
    \|\mathcal{R}\proj_{\neq 0}(\nabla a_{\xi}^2\proj_{\geq \lambda\sigma/2}( \et^2))\|_{\LinfL}\lesssim C_a(5)\frac{\log^3(\lambda \sigma)}{\lambda \sigma }\log^2(\lambda).
\end{equation*}
Instead,
\begin{equation*}
    -\proj_{\neq 0}\Bigl(a_{\xi}^2\partial_t\et^2\frac{\xi}{\mu}\Bigr)+P_H\proj_{\neq 0}\Bigl(a_{\xi}^2 \et^2 \frac{\xi}{\mu}\Bigr)=
    -(I-P_H)\Bigl(\proj_{\neq 0}\partial_t(a_{\xi}^2\et^2) \frac{\xi}{\mu}\Bigr)+\proj_{\neq 0}\Bigl(\frac{\xi}{\mu}\partial_t(a_{\xi}^2)\et^2\Bigr),
\end{equation*}
where the first term goes into the pressure thanks to $I-P_H=\nabla \Delta^{-1}\Div$, while the second one can be seen as part of $R_{q+1}$, and
\begin{equation*}
    \bigl\|\mathcal{R}\proj_{\neq 0}\bigl(\tfrac{\xi}{\mu}\partial_t(a_{\xi}^2)\et^2\bigr)\bigr\|_{\LinfL}\lesssim C_a(5)\frac1{\mu}\log^2(\lambda).
\end{equation*}

\subsubsection{Conclusion of the proof}
Summing up all the contributions to $R_{q+1}$, we end up with
\begin{equation}\label{Rq+1Est}
\begin{aligned}
  |R_{q+1}|_{\LinfL}
    \lesssim & C_a(1) \big(\frac{1}{\lambda}+\sigma r + \frac{r^{\frac{3}{2}}}{\mu}\big)\big(\frac{1}{\lambda}+\sigma r + \frac{r^{\frac{3}{2}}}{\mu}+\delta_{q+1}^{\frac{1}{2}}+(\lambda \sigma)^{-\frac{1}{2}}\big)\\
    & + C_a(0)r^{\frac{3}{2}-\frac{3}{s}}+C_a(1)\sigma \mu r^{-\frac{1}{2}}\log(\lambda)+C_a(2)\big(\frac{\lambda}{r}\big)^{\frac{3}{2}}\log^{4-\beta}(\lambda)\\
    &+C_a(2)\frac{(\lambda \sigma r)^{\frac{3}{2}}}{\mu}\log(\lambda)
      +C_a(5)\frac{\log^3(\lambda \sigma)}{\lambda \sigma}\log^2(\lambda)\\
    &+C_a(5)\log^5(\lambda)\sigma r+C_a(5)\frac{\log^2(\lambda)}{\mu}.
\end{aligned}
\end{equation}
Since the values of $C_a(N)$ do not depend on the parameter $\lambda$ (see \cref{Lemma1}), then we can choose the parameters $s$, $\sigma$, $\mu$ and $r$ such that \cref{AssumPams} hold and $\|R_{q+1}\|_{\LinfL}$ is smaller than $\delta_{q+2}$ as long as $\lambda$ is large enough.
Therefore, consider
\begin{equation*}
    s=\frac{3}{2},\quad \sigma=\frac1\lambda\log^2(\lambda) \log^4(\log(\lambda)),\quad r=\frac{\lambda}{\log^{y}(\lambda)},\quad \mu=\lambda^{\frac{3}{2}}\log^{z}(\lambda),
\end{equation*}
where $y>2$ and $z>-\frac{3}{2}y$ in order to satisfy \cref{AssumPams}.
In particular, by this choice, the RHS of \cref{Rq+1Est} tends to zero as long as $\lambda$ diverges to $\infty$ if the following inequalities hold
\begin{gather*}
    z>-\frac{3}{2}y,\qquad z+\frac{y}{2}<-3,\qquad y>7,\\
    \frac{3}{2}y+4-\beta <0,\qquad \frac{3}{2}y+z>4.
\end{gather*}
It is possible to find a pair $(y,z)$ that satisfies the previous inequality if and only if $\beta >\frac{29}{2}$.
In order to conclude the proof of \cref{IterLemma}, we need to show that \cref{IterLemma1} and \cref{IterLemma2} are fully satisfied by our construction and our choice of the parameters. In particular, the control on the temporal support follows directly from \cref{TempSupp}, while the estimates on $\|v_{q+1}-v_q\|_{\LinfLSec}$ and $\||\nabla|^{\frac{3}{2}}T_m(v_{q+1}-v_q)\|_{\LinfL}$ are due to \cref{MainEstimate} and \cref{linearEst}.

This concludes the proof of the \emph{Iteration Lemma} (\cref{IterLemma}).

\bibliographystyle{amsalpha}
\bibliography{refs}
\end{document}